\documentclass[a4paper,12pt]{amsart}
\usepackage[utf8]{inputenc}
\usepackage{amsmath}
\usepackage{amsfonts}
\usepackage{amsthm}
\usepackage{mathtools}
\usepackage{amssymb}
\usepackage[mathscr]{euscript}
\usepackage{xcolor}
\usepackage{enumerate}
\usepackage{tikz}
\usepackage{tikz}
\usepackage{bbm}
\usepackage{marvosym}
\usepackage{skull}
\usepackage[utf8]{inputenc}
\usepackage{xcolor}
\usepackage[all]{xy}
\usepackage{tikz-cd}
\usepackage[hidelinks]{hyperref}
\usepackage{cite}
\usepackage{soul}

\usepackage[a4paper, left=1in, right=1in, top=1in, bottom=1.5in]{geometry}

\usepackage{amsthm}
\newtheorem{theorem}{Theorem}[section]
\newtheorem{lemma}[theorem]{Lemma}

\newtheorem{example}[theorem]{Example}
\newtheorem{remark}[theorem]{Remark}
\theoremstyle{definition}

\usepackage{xcolor}
\let\oldcolor\color
\renewcommand{\color}[1]{\oldcolor{black}}

\newtheorem{proposition}[theorem]{Proposition}
\newtheorem{corollary}[theorem]{Corollary}

\theoremstyle{definition}
\newtheorem{definition}[theorem]{Definition}

\theoremstyle{remark}

\numberwithin{equation}{section}

\newcommand{\m}{{}^{-1}}

\newcommand{\dom}{{\rm dom}}

\newcommand{\Z}{\mathbb{Z}}

\newcommand{\R}{{\mathbb R}}

\begin{document}
	
	\title{On $\Gamma$-embeddings and partial actions on function spaces}
	
	\author{Luis A. Mart\'inez-Sánchez, H\'ector Pinedo and Jose L. Vilca-Rodríguez
		\\[3pt] }
	
	\address{Departamento de  Matem\'aticas,
		Facultad de Ciencias, Universidad Nacional Aut\'onoma de M\'exico, Mexico City,  Mexico}
	\email{luchomartinez9816@hotmail.com}
	
	\address{Departamento de  Matem\'aticas,
		Universidad Industrial de Santander, Bucaramanga, Colombia}
	\email{hpinedot@uis.edu.co}

	\address{Intituto de Matemática e Estatística, Universidade de São Paulo, São Paulo, Brazil}
	\email{jvilca@ime.usp.br}
	% General info

	%\thanks{The first named author was supported by grant 829945 from SECIHTI (M\'exico). The third named author was supported by PRPI da Universidade de São Paulo, process n°: 22.1.09345.01.2. }
	
	\begin{abstract}  
		This paper deals with the extension of partial actions of topological groups on topological spaces. Within this framework, we introduce a class of topological embeddings defined via the inverse semigroup of homeomorphisms between open subsets of a topological space. We describe several embeddings of this type, referred to as $\Gamma$-embeddings, and we place particular emphasis on one of them. In particular, we prove that every topological space $Y$ admits a $\Gamma$-embedding into the space of continuous functions $C(X,Y)$, equipped with the compact-open topology, where $X$ is a compact space. Consequently, any partial action $\theta$ of a topological group $G$ on $Y$ naturally induces a partial action $\hat{\theta}$ on $C(X,Y)$. Throughout the paper, we investigate various relationships between these actions, as well as between their corresponding globalizations and enveloping spaces.
	\end{abstract}
	
	\maketitle\markboth{Luis A. Martínez-Sánchez, Héctor Pinedo and Jose L. Vilca-Rodríguez}{Partial actions on  the spaces of continuos functions and $\Gamma$-embeddings}

\section {Introduction}
	
	Given an action $\beta : G \times X \to X$ of a group $G$ on a set $X$ and an invariant subset $Y \subset X$ (that is, $\beta(g,y)\in Y$ for all $y\in Y$ and $g\in G$), the restriction of $\beta$ to $G\times Y$ yields an action of $G$ on $Y$. In contrast, when $Y$ is not invariant, one obtains what is known as a partial action on $Y$. More precisely, a partial action is given by a family of bijections $\{\theta_g\}_{g\in G}$ defined on subsets of $Y$ such that $\theta_e=\mathrm{id}_Y$ and, for all $g,h\in G$, the composition $\theta_g\circ\theta_h$ is a restriction of $\theta_{gh}$, where $\theta_g\circ\theta_h$ is defined on the largest possible domain on which this composition makes sense.
	
	To the best of our  knowledge,  a systematic study of partial actions (as we know them now) first appeared in the doctoral thesis of Palais~\cite{Palais1957}, where they arose as maximal  local actions of Lie groups on smooth manifolds. Subsequently, partial group actions  were used in the theory of $C^*$-algebras, providing a fundamental tool to characterize important classes of algebras, as  more general crossed products. The first result in this direction  was obtained  by  Exel in~\cite{exel1994}. In addition, Kellendonk and Lawson~\cite{KL} pointed out the importance of partial actions to several branches of mathematics, including $\mathbb{R}$-trees, Fuchsian groups, presentations of groups, model theory and other topics. Consequently, partial actions have been studied by different authors in several areas. For a comprehensive overview of their development, we refer the reader to the surveys by Dokuchaev~\cite{dokuchaev2011,dokuchaev2019} and to Exel’s monograph~\cite{RB}.
	
	A relevant problem is whether a given partial action can be obtained as the restriction of a corresponding collection of total maps on some bigger space (see Example \ref{induc}). In the topological framework, this problem was first investigated by Abadie \cite{AB} and independently	by Kellendonk and Lawson \cite{KL}. They showed that for any topological partial action of a topological group $G$	on a space $Y$, there exist a  space $X$ and a continuous action $\beta$ of $G$ on $X$ such that $Y$ is a subspace of $X$ and the partial action $\theta$ coincides with the restriction of $\beta$ to $Y$. Such a space $X$ is called a globalization of $Y$. It was also shown	that there is a minimal globalization,  called the enveloping space. However, structural properties of $Y$ need not be preserved under globalization. For instance, Abadie~\cite{AB} proved that the enveloping action of a partial action defined on a Hausdorff space may fail to act on a Hausdorff space (see Proposition 1.2 and Example 1.4 in  \cite{AB}). This observation motivates the problem of identifying which structural properties are inherited by globalizations; further progress in this direction can be found in several works, for example  in \cite{RGG, MP, MPR, MPV1, MPV2, PU1, S}.

	On the other hand,  topological properties of spaces of continuous functions have been extensively studied, and recently group actions on them have drawn the attention of several researchers,  \cite{BC, SAH, SAH2, R}. Thus, in this work, we focus on partial actions of continuous maps with compact domains. One of our motivations comes from the well-known result  that  establishes a  correspondence between partial actions on a Hausdorff locally compact space $X$ and the partial actions on the $C^*$-algebra $C_0(X)$, consisting of  continuous functions with values in $\mathbb{R}$  that vanish at infinity. This result  was obtained in  \cite{AB2}, and  allowed the author to obtain a  groupoid endowed with a Haar system whose $C^*$-algebra agrees.  Furthermore, in  \cite{GRO} the authors used the ring of continuous maps $C(X)$  as a tool to relate topological partial actions of $X$ with algebraic partial actions on $C(X),$ obtaining important relations between certain structural  properties  of  the space $X$  and  the ring $C(X).$

	We now outline the structure of the article. After this introduction, Section~\ref{pre} presents the basic notions of topological partial actions and globalizations. It is worth noting that a classical problem concerns the extension of group actions to larger spaces. For example, the classical Mackey \& Horth theorem  asserts that if $H$ is a Polish group and $G\subset H$ is a closed subgroup, then any action of $G$ on a Polish space $X$ can be extended to an action of $H$ on a Polish space $Y$ such that $X$ is closed in $Y$. In the realm of partial actions, an analogue of this result was obtained in~\cite[Theorem~4.6]{PU1}. With this motivation in mind, after presenting the notion of topological partial actions in terms of unital premorphisms, we introduce in Definition~\ref{cemb} the notion of a $\Gamma$-embedding.
	
	We then show in Proposition~\ref{homosemi1} that any topological space $X$ is $\Gamma$-embedded in the hyperspace $\mathcal K(X)$ of all nonempty compact subsets of $X$ endowed with the Vietoris topology. Moreover, Proposition~\ref{cone} establishes that $X$ is also $\Gamma$-embedded in its cone $\mathrm{Cone}(X)$. Using Proposition~\ref{motiva1}, partial actions can be induced from $X$ to each of these spaces; in the case of $\mathcal K(X)$, this construction was previously considered by the first two authors in~\cite{MPR}. Given the relevance of spaces of continuous maps to the theory of partial actions, Section~\ref{pafun} is devoted to the study of partial actions on $C(X,Y)$, where $X$ is a compact space and $Y$ is a topological space. Motivated by the proof of Proposition~\ref{homosemi1}, we prove in Theorem~\ref{YencajaenC(X,Y)} that the space $Y$ is $\Gamma$-embedded in $C(X,Y)$. This result allows us to construct, in a natural way, a topological partial action $\hat{\theta}$ of a group $G$ on $C(X,Y)$ from a given topological partial action $\theta$ of $G$ on $Y$, as described in equation~\eqref{eq paC(x,y)}.
	
	Furthermore, Propositions~\ref{continuidad} and~\ref{nice} relate structural properties of $\theta$ and $\hat{\theta}$, such as continuity and the property of being nice. Subsection~\ref{globali} is devoted to the study of separability properties of the corresponding enveloping spaces $Y_G$ and $C(X,Y)_G$. In particular, Theorem~\ref{embed} shows that, assuming $\theta$ is a nice partial action, the space $C(X,Y)_G$ is $G$-homeomorphic to an open subset of $C(X,Y_G)$. This result enables us to relate regularity and metrizability properties of the spaces $Y_G$ and $C(X,Y)_G$ in Corollary~\ref{regme}. Finally, Subsection~\ref{examp} presents an example in which the spaces $C(X,Y)_G$ and $C(X,Y_G)$ are $G$-homeomorphic. The article concludes with a study of the ANE condition, where we show that this property is equivalent for each of the spaces $Y$, $C(X,Y)$, $Y_G$, and $C(X,Y)_G$. Several examples are provided to illustrate our results.

\section{Partial actions and globalization}\label{pre} 

Throughout this work, the letter $G$ will denote a group with identity element $e$. If $G$ is a topological group, we suppose that $G$ is Hausdorff.

Let $X$ be a set. Consider a partially defined map  $$\theta: G*X\subset G\times X\ni (g,x)\mapsto \theta(g,x)=g\cdot x\in  X.$$ 
Following the convention in \cite{KL}, we write $\exists g\cdot x$  if  $(g,x)\in G*X$. We say that $\theta$ is a (set-theoretic) partial action of $G$ on $X$ if, for every $g,h \in G$ and $x \in X$, the following conditions hold:	

\begin{enumerate}
	\item [(PA1)] If $\exists\ g\cdot x$, then $\exists\ g^{-1}\cdot (g\cdot x)$ and          $g^{-1}\cdot (g\cdot x)=x$.\vspace{0.2cm}
	\item [(PA2)] If $\exists\ g\cdot(h\cdot x)$, then $\exists\ (gh)\cdot x$ and             $g\cdot(h\cdot x)=(gh)\cdot x$.\vspace{0.2cm}
	\item [(PA3)] $\exists\ e\cdot x$ and $e\cdot x=x.$
\end{enumerate}
The partial action  $\theta$ is  called  {\it global} if $G*X= G\times X$.  

Let $\theta:G*X\subset G\times X\to  X$ be a partial action of $G$ on $X$.  For each $g\in G$, define $$X_g=\{x\in X\mid \exists\, g\m \cdot x\},$$
and consider the mapping $$\theta_g\colon X_{g\m}\ni x\mapsto g\cdot x\in X_g.$$
Then $\theta_g$ is a bijection, and the family 
	
\begin{equation}\label{family}
	\{\theta_g\colon X_{g\m}\to X_g \}_{g\in G},
\end{equation}
satisfies:
	
\begin{enumerate}
	\item [(PA1')] $X_e=X$ and $\theta_e={\rm id}_X.$\vspace{0.2cm}
	\item [(PA2')] \( \theta_g(X_{g^{-1}} \cap X_{g^{-1}h}) \subset X_g \cap X_h \) for all \( g, h \in G \), and 
        \( \theta_g \circ\theta_h(x) = \theta_{gh}(x) \) for all \( x \in X_{h^{-1}} \cap X_{h^{-1}g^{-1}} \).
	
\end{enumerate}

Conversely, let $\theta=\{\theta_g \colon X_{g^{-1}} \to X_g\}_{g \in G}$ be a family of bijections satisfying (PA1') and (PA2').  
Define
\[
G * X = \{ (g,x) \in G \times X \mid x \in X_{g^{-1}} \}.
\]

Then one readily verifies that the partial mapping
\[
\theta: G*X\subset G\times X\ni (g,x)\mapsto \theta_g(x)\in X 
\]
defines a partial action of $G$ on $X$.
Moreover,  $\theta$  is  global if and only if $X_g=X$ for every $g\in G.$

Let $G$ be a topological group, $X$ a topological space, and $\theta: G*X\rightarrow X$ a set-theoretic partial action of $G$ on the underlying set $X$. We endow $G\times X$ with the product topology and $G*X$ with the subspace topology.  In this setting, several types of partial actions have been introduced; we highlight below those that are particularly relevant to this work.

We say that a partial action $\theta:=\{\theta_g\colon X_{g\m}\ni x\mapsto g\cdot x\in X_g \}_{g\in G}$ is:
	
	\begin{enumerate}[(i)]
		\item   {\it topological} if $X_g$ is open in $X$ and $\theta_g: X_{g\m}\rightarrow X_g$ is continuous for each $g\in G$;
		\item  {\it continuous } if  $\theta: G*X\to X$ is a continuous map;\vspace{0.2cm}
		\item   {\it nice } if $\theta$ is continuous and $G*X$ is open in $G\times X.$
	\end{enumerate}
	
	Throughout this paper, all partial actions are assumed to be topological. 
	
	A natural way to construct partial actions from global actions is as follows.

\begin{example}[{\bf Induced partial action}]\label{induc}	{\rm
	Let $\mu \colon G\times Z\to Z$ be a continuous global action of a topological group $G$ on a  topological space $Z$, and let $X$ be a nonempty  subset of $Z$.  A  continuous partial action of $G$ on $X$ is obtained by setting:
	\begin{equation*}\label{induced}
		G*X=\{(g,x)\in G\times X\mid \mu(g,x)\in X\}\quad  \text{ and }\quad \theta= \mu\restriction_{G*X}.
	\end{equation*}

Moreover,  $\theta$ is a nice partial action whenever $X$ is open in $Z$. We say that   $\theta$ is  the {\it restriction of $\mu$ to $X$}}.
\end{example}

As mentioned in the introduction, a positive answer to the problem of determining whether a given partial action arises as a restriction of a global action was established in \cite[Theorem 1.1]{AB} and independently in \cite[Section 3.1]{KL}. For the reader's convenience, we recall their construction below.
	
Let $\theta$ be a topological partial action of $G$ on a topological space $X$, and let us consider the following equivalence relation on $ G\times X$:
	\begin{equation}\label{equiv}
		(g,x)R(h,y) \Longleftrightarrow x\in X_{g\m h}\quad \text{and}\quad \theta_{h\m g}(x)=y.
	\end{equation}
	
The equivalence class of $(g,x)$ is denoted by $[g,x]$. The {\it enveloping space}  or the {\it globalization}  of $X$ is the set  $X_G=(G\times X)/R$,  endowed with the quotient topology. By \cite[Proposition 3.9 (iii)]{KL}  the map
	\begin{equation}
		\label{action}
		\mu^\theta \colon G\times X_G\ni (g,[h,x])\to [gh,x]\in X_G,
	\end{equation}
	is a continuous action of $G$ on $X_G$, which is called the {\it enveloping action} of $\theta$. 
	
	The quotient map
	\begin{equation} \label{qo} q^\theta: G\times X\ni (g,x)\mapsto [g,x]\in X_G,
	\end{equation}
	is continuous and  open by \cite[Proposition 3.9 (ii)]{KL}.  Moreover, it follows from  \cite[Proposition 3.9 (i)]{KL} that the map  
	\begin{equation}\label{iota}
		\iota^\theta \colon X\ni x\mapsto [e,x]\in X_G
	\end{equation} 
	is continuous, injective, and $G\cdot \iota(X)=X_G$. 
	
	When no confusion arises, we will omit $\theta$ in (\ref{action}), (\ref{qo}) and (\ref{iota}).
	
	A  {\it $G$-space} is  a pair $(X,\theta)$, where $X$ is a topological space and $\theta$ is a topological  partial action of $G$ on $X$. Given a $G$-space $(Y,\eta)$, a continuous map $f: X\rightarrow Y$ is  a {\it $G$-map} (or an equivariant map) when $(g,f(x))\in G*Y$ and $\eta(g,f(x))=f(\theta(g,x))$, for each $(g,x)\in G*X$. Furthermore, $f$ is a {\it $G$-homeomorphism} if it is a homeomorphism and $f^{-1}$ is also a $G$-map. In this case, the spaces $X$ and $Y$ are called {\it $G$-homeomorphic}.

We now review the relationship between partial group actions and semigroup theory. To this end, we start by recalling the following definition.
	
\begin{definition} 
	Let $X$ be a topological space, and let $\Gamma(X)$ denote the semigroup of homeomorphisms between open subsets of $X$. The multiplication of two maps $\psi$ and $\phi$ in $\Gamma(X)$ is defined by the composition:
$$\psi\circ \phi: \phi^{-1}(\rm {dom}(\psi) \cap \rm {im} (\phi))\to \psi(\rm {dom}(\psi) \cap \rm {im} (\phi)). $$ 
\end{definition}

Note that $\Gamma(X)$ is  an inverse monoid with unit  ${\rm id}_X.$ Moreover, $\Gamma(X)$ is a poset under the natural partial order $\leq$, defined by
\[
\psi \leq \phi \quad \text{if and only if} \quad \psi \text{ is a restriction of } \phi.
\]
According to \cite[p. 88]{KL}, a collection $\theta=\{\theta_g\}_{g\in G}\subset \Gamma(X)$ is a topological partial action of $G$ on $X$ if and only if the mapping $\theta: G\ni g\mapsto \theta_g\in \Gamma(X)$, which we also denote by $\theta$, is a {\it unital premorphism}, that is, $\theta$ satisfies the following conditions:

\begin{center}
	$\theta_e={\rm id}_X,\quad \theta_{g^{-1}}=\theta_g^{-1},\quad \theta_g\circ \theta_h\leq \theta_{gh},$
\end{center}
for all $g,h\in G$. Hence, we shall use the terms topological partial action of $G$ on $X$ and unital premorphism from $G$ to $\Gamma(X)$ interchangeably.

\subsection{Embeddings  of partial actions}

Given a topological space $X$, we investigate the problem of finding a topological space $Z$ and an embedding $c \colon X \to Z$ such that every topological partial action $\theta \colon G*X \to X$  extends to a topological partial action $\hat{\theta} \colon G*Z \to Z$ for which $c$ is a $G$-map.

For any topological space $X$, we  consider the set
\begin{center}
	$\Gamma(X)*X=\{(f,x)\in \Gamma(X)\times X\mid x\in \dom(f)\}$
\end{center}
 and the function
 \begin{center}
 	$\text{ev}^X:\Gamma(X)*X\ni (f,x)\mapsto f(x)\in X.$
 \end{center}

We now introduce a key definition for our purposes.

\begin{definition}\label{cemb}
	Let $X$ and $Z$ be topological spaces. A \emph{$\Gamma$-embedding} of $X$ into $Z$ is a pair $(\sigma,c)$ consisting of a semigroup homomorphism $\sigma:\Gamma(X)\to \Gamma(Z)$ and a topological embedding $c:X\to Z$ such that:
	
	\begin{enumerate}
		\item [(i)] The function 
		\begin{center}
			$\sigma\times c:\Gamma(X)*X\ni (f,x)\mapsto  \big(\sigma(f),c(x)\big)\in \Gamma(Z)\times Z,$
		\end{center}
		satisfies  $\mathrm{im}(\sigma\times c)\subset \Gamma(Z)*Z$.
		\item [(ii)] The  diagram
		\begin{center}
				$\xymatrix{\Gamma(X)*X\ar[d]_-{\text{ev}^X}\ar[r]^-{\sigma\times c}& \Gamma(Z)*Z\ar[d]^-{\text{ev}^Z}\\
					X\ar[r]_-{c}& Z
				}$
		\end{center}
		commutes.
	\end{enumerate}
	
	We say that $X$ is \emph{$\Gamma$-embedded in $Z$} if there exists a $\Gamma$-embedding of $X$ into $Z$.
\end{definition}

\begin{proposition}\label{motiva1}
	\rm{Let $X$ and $Z$ be two topological spaces. Suppose that there exists a $\Gamma$-embedding $(\sigma,c)$ of $X$ into $Z$ such that $\sigma(\mathrm{id}_X)=\mathrm{id}_Z$. Then any topological partial action $\theta:G\to \Gamma(X)$ induces a topological partial action $\hat{\theta}=\sigma\circ\theta: G\to \Gamma(Z)$ of $G$ on $Z$ such that $c$ is a $G$-map.}
\end{proposition}

\begin{proof}
	Since $\sigma$ is a semigroup homomorphism and  $\sigma(\mathrm{id}_X)=\mathrm{id}_Z$, the function $\hat{\theta}=\sigma\circ\theta:G\rightarrow \Gamma(Z)$ is a unital premorphism, and hence a topological partial action of $G$ on $Z$. It remains to show that $c$ is a $G$-map.  To this end, let $g\in G$ and $x\in X$ be such that $\exists\ g\cdot x$. Then $(\theta(g),x)\in \Gamma(X)*X$. In addition,
		\begin{center}
			$\big(\hat{\theta}(g),c(x)\big)=\big(\sigma(\theta(g)),c(x)\big)= (\sigma\times c)\big(\theta(g),x\big)\in\text{im}(\sigma\times c)\subset \Gamma(Z)*Z$
		\end{center}
	which implies that $c(x)\in\text{dom}(\hat{\theta}(g))$. Moreover,
		\begin{center}
			$\hat{\theta}(g)\big(c(x)\big)=\text{ev}^Z\big(\hat{\theta}(g),c(x)\big)=\big(\text{ev}^Z\circ(\sigma\times c)\big)(\theta(g),x)=\big(c\circ \text{ev}^X\big)(\theta(g),x)=c\big(\theta(g)(x)\big)$
		\end{center}
	showing that $c$ is a $G$-map. This concludes the proof.
\end{proof}

\subsection{Some examples of  $\Gamma$-embeddings} We now show how a topological space $X$ can be  $ \Gamma$-embedded in various topological spaces.
	
Let $Z$ be a topological space and   $c:X\to Z$  be  an open embedding. It is not difficult to verify that  $X$ is $\Gamma$-embedded in $Z,$ where the function $\sigma: \Gamma(X)\ni f\mapsto \sigma(f)\in \Gamma(Z)$ is defined by
	 \begin{center}
	 	$\sigma(f): c(\dom (f))\ni c(x)\mapsto  c(f(x))\in c(\mathrm{im} (f)),\quad$ for all $f\in \Gamma(X)$.
	 \end{center}
	 
Consequently, using  \cite[Theorem 3.13 (i)]{KL},  we relate   the concepts of $\Gamma$-embeddings and globalization, as stated in the following result.
	
\begin{proposition}
	Let $X$ be a topological space and $\theta:G*X\to X$ be a nice partial action. Then $X$ is $\Gamma$-embedded in $X_G,$ where  $c=\iota^\theta$ and $\sigma: \Gamma(X)\ni f\mapsto \sigma(f)\in \Gamma(X_G)$ is defined by  
	$$\sigma(f): \iota^\theta(\dom (f))\ni \iota^\theta(x)\mapsto \iota^\theta(f(x))\in \iota^\theta(\mathrm {im} (f)).$$
\end{proposition}
\medskip
	
Consider the hyperspace $\mathcal{K}(X)$ of all nonempty compact subsets of $X$, endowed with the Vietoris topology. This topology has as a subbasis the following sets
	\begin{center}
		$V^-=\{A\in \mathcal{K}(X)\mid A\cap V\neq \emptyset\}$\,\,\, and \,\,\, $V^+=\{A\in \mathcal{K}(X)\mid A\subset V\}$,
	\end{center}
for each nonempty open subset $V$ of $X$. 

It is well known that the map 
\begin{equation}\label{ckx}c: X\ni x \mapsto  \{x\}\in  \mathcal{K}(X)
\end{equation}
is an embedding.

For each $f\in \Gamma(X)$, define $\sigma(f)\in \Gamma(\mathcal{K}(X))$ as follows. If $\text{dom}(f)\neq\emptyset$, set  $$\sigma(f):\text{dom}(f)^+\ni A \mapsto f(A)\in  \text{im}(f)^+.$$ If $f=1_\emptyset$,  define $\sigma(f)=1_\emptyset$.

\begin{proposition}\label{homosemi1}
	\rm{Any topological space $X$ is $\Gamma$-embedded in $\mathcal{K}(X)$.}
\end{proposition}
	
\begin{proof}
	In order to see that $\sigma:\Gamma(X)\rightarrow \Gamma(\mathcal{K}(X))$ is a semigroup homomorphism, take $f,g\in \Gamma(X)$, and let us show that $\sigma(f\circ g)=\sigma(f)\circ \sigma(g)$. 
	First, assume that $\mathrm{dom}(f\circ g)=\emptyset$. If $ \mathrm{dom}(\sigma(f)\circ \sigma(g))\neq\emptyset$, then there exists $A\in \mathcal{K}(X)$ such that $A\subset \dom(g)$ and $g(A)\subset \mathrm{dom}(f)$, which implies $A\subset \mathrm{dom}(f\circ g)$, a contradiction. Hence, $\mathrm{dom}(\sigma(f)\circ\sigma(g))=\emptyset$ and $\sigma(f\circ g)=\sigma(f)\circ\sigma(g)$. 
	
	Now suppose that $\mathrm{dom}(f\circ g)\neq\emptyset$. Then
		\begin{align*}
			\mathrm{dom}(\sigma(f\circ g))&=\{A\in \mathcal{K}(X)\mid A\subset \mathrm{dom}(f\circ g)\}\\
			&=\{A\in \mathcal{K}(X)\mid A\subset \mathrm{dom}(g)\ \wedge\ g(A)\subset \mathrm{dom}(f) \}\\
			&=\{A\in \mathcal{K}(X)\mid A\in \mathrm{dom}(\sigma(g))\ \wedge\ \sigma(g)(A)\in \mathrm{dom}(\sigma(f))\}\\
			&=\mathrm{dom}(\sigma(f)\circ\sigma(g)).
		\end{align*}
	Moreover, it is straightforward to verify that $\sigma(f\circ g)(A)=(\sigma(f)\circ \sigma(g))(A),$ for each $A\in\mathrm{dom}(\sigma(f\circ g))$. Thus, $\sigma$ is a semigroup homomorphism. 
	
	Finally, observe that for each $(f,x)\in\Gamma(X)*X$, we have $x\in\mathrm{dom}(f)$, and hence  $c(x)=\{x\}\in\mathrm{dom}(\sigma(f))$. Therefore, the map $$\sigma\times c: \Gamma(X)*X\rightarrow \Gamma(\mathcal{K}(X))\times \mathcal{K}(X)$$ satisfies $\mathrm{im}(\sigma\times c)\subset \Gamma(\mathcal{K}(X))*\mathcal{K}(X)$. Furthermore, it is clear that the diagram
		
	\begin{center}
			$\xymatrix{\Gamma(X)*X\ar[d]_-{\mathrm{ev}^X}\ar[r]^-{\sigma\times c}& \Gamma(\mathcal{K}(X))*\mathcal{K}(X)\ar[d]^-{\text{ev}^{\mathcal{K}(X)}}\\
				X\ar[r]_-{c}& \mathcal{K}(X)
			}$
	\end{center}
commutes. Hence,  the space $X$ is $\Gamma$-embedded in $\mathcal{K}(X)$.
	\end{proof}
	
It is not difficult to see that the semigroup homomorphism  $\sigma:\Gamma(X)\rightarrow \Gamma(\mathcal{K}(X))$ satisfies $\sigma(\mathrm{id}_X)=\mathrm{id}_{\mathcal{K}(X)}$. Hence,	according to Proposition \ref{motiva1},  we have that given a topological space $X$, every topological partial action $\theta:G\rightarrow \Gamma(X)$ induces a topological partial action $\hat{\theta}:G\rightarrow \Gamma(\mathcal{K}(X))$ defined by $\hat{\theta}=\sigma\circ\theta$. Moreover
	\begin{align*}
		G*\mathcal{K}(X)=&\{(g,A)\in G\times \mathcal{K}(X)\mid A\in\mathrm{dom}(\hat{\theta}(g))\}\\
		=&\{(g,A)\in G\times \mathcal{K}(X)\mid A\in\mathrm{dom}(\theta(g))^+\}\\
		=&\{(g,A)\in G\times \mathcal{K}(X)\mid A\subset\mathrm{dom}(\theta(g))\},
	\end{align*}
	and $\hat{\theta}(g)(A)=\sigma(\theta(g))(A)=\theta(g)(A)$ for all $(g,A)\in G*\mathcal{K}(X)$. This partial action was studied by the first and second named authors in \cite{MPR}.\medskip

	Recall that  the {\it cone of X} is the topological space  $$\mathrm{Cone}(X):=([0,1]\times X)/(\{0\}\times X),$$ endowed with the quotient topology. We denote by $\eta: [0,1]\times X\rightarrow \mathrm{Cone}(X)$ the quotient map, and by $tx$ the equivalence class of $(t,x)\in [0,1]\times X$, that is, $\eta(t,x)=tx$.

	\begin{proposition}\label{cone}
		\rm{Let $X$ be a topological space. Then $X$ is $\Gamma$-embedded in $\mathrm{Cone}(X)$.}
	\end{proposition}
	
	\begin{proof}
	It is not difficult to see that the map $c:X\rightarrow \mathrm{Cone}(X)$ defined by $c(x)=1x$ for all $x\in X,$ is an embedding.  
	To construct a semigroup homomorphism $\sigma: \Gamma(X)\rightarrow \Gamma(\mathrm{Cone}(X)),$ let  $f\in\Gamma(X)$. Observe that the sets  $\eta((0,1]\times \mathrm{dom}(f))$ and $\eta((0,1]\times \mathrm{im}(f))$ are open subsets of $\mathrm{Cone}(X)$. Define  $$\sigma(f): \eta((0,1]\times \mathrm{dom}(f))\ni tx\mapsto tf(x)\in \eta((0,1]\times \mathrm{im}(f)).$$ The map $\sigma(f)$ is continuous and is a homeomorphism with inverse $\sigma(f^{-1})$. Hence, we obtain a well-defined function $\sigma \colon \Gamma(X)\to \Gamma(\mathrm{Cone}(X))$. Moreover, for each   $f, g\in \Gamma(X)$, we have
		\begin{align*}
			\mathrm{dom}(\sigma(f\circ g))=\ & \{tx\in \mathrm{Cone}(X)\mid t>0\ \wedge\ x\in \mathrm{dom}(f\circ g)\}\\
			=\ & \{tx\in \mathrm{Cone}(X)\mid t>0\ \wedge\ x\in \mathrm{dom}( g)\ \wedge\ g(x)\in \mathrm{dom}(f)\}\\
			=\ & \{tx\in \mathrm{Cone}(X)\mid tx\in \mathrm{dom}(\sigma(g))\ \wedge\ g(x)\in \mathrm{dom}(f)\}\\
			=\ & \{tx\in \mathrm{Cone}(X)\mid tx\in \mathrm{dom}(\sigma(g))\ \wedge\ \sigma(g)(tx)\in \mathrm{dom}(\sigma(f))\}\\
			=\ & \mathrm{dom}(\sigma(f)\circ \sigma(g)).
		\end{align*}
		
	Furthermore, for every $tx\in \mathrm{dom}(\sigma(f)\circ\sigma(g))$,
		\begin{center}
			$(\sigma(f)\circ \sigma(g))(tx)=\sigma(f)(tg(x))=tf(g(x))=\sigma(f\circ g)(tx).$
		\end{center}
		
Thus, $\sigma(f)\circ \sigma(g)=\sigma(f\circ g)$, and therefore $\sigma$ is a semigroup homomorphism.
		
	Finally, it is clear that $(\sigma\times c)(\Gamma(X)*X)\subset \Gamma(\mathrm{Cone}(X))*\mathrm{Cone}(X)$, and the following diagram is commutative
			\begin{center}
			$\xymatrix{\Gamma(X)*X\ar[d]_-{\mathrm{ev}^X}\ar[r]^-{\sigma\times c}& \Gamma(\text{Cone}(X))*\mathrm{Cone}(X)\ar[d]^-{\text{ev}^{\text{Cone}(X)}}\\
				X\ar[r]_-{c}& \text{Cone}(X).
			}$
		\end{center}
		Hence, the space $X$ is $\Gamma$-embedded in $\mathrm{Cone}(X)$.
	\end{proof}
	
	We finish this section with the next proposition.
	
	\begin{proposition}\label{prod}
		\rm{Let $\{X_i \}_{1\leq i\leq n}$ be a family of topological spaces, and let $X=\prod_{i=1}^n X_i$ be endowed with the product topology. Then $X_i$ is $\Gamma$-embedded in $X$ for each $i\in\{1,\dots,n\}$.}
	\end{proposition}
	
	\begin{proof}
	Fix $i\in\{1,\dots,n\}$. For each $j\in\{1,\dots,n\}\setminus\{i\}$, choose an element $z_j\in X_j$. Define the map $$c: X_i\ni x\mapsto (z_1,\dots,z_{j-1},x,z_{j+1},\dots,z_n)  \in   X.$$  
	Clearly, $c$ is an embedding.  
	
	Now, given  $f\in \Gamma(X_i)$, define
	\begin{center}
		$\sigma(f):\pi_i^{-1}(\mathrm{dom}(f))\rightarrow \pi_i^{-1}(\mathrm{im}(f)),\quad$ $\sigma(f)(x_1,\dots,x_i,\dots,x_n)=(x_1,\dots,f(x_i),\dots,x_n)$.
	\end{center} 
	
	Then $\sigma(f)\in \Gamma(X)$, and hence we have  a function $\sigma: \Gamma(X_i)\rightarrow \Gamma(X)$. We claim that $\sigma$ is a semigroup homomorphism. Indeed, for each $f,g\in \Gamma(X_i)$, we have
		\begin{align*}
		\mathrm{dom}(\sigma(f)\circ\sigma(g))=\ & \sigma(g)^{-1}(\mathrm{im}(\sigma(g))\cap\mathrm{dom}(\sigma(f)))\\
		=\ & \sigma(g\m)(\pi_i^{-1}(\mathrm{im}(g))\cap\pi_i^{-1}(\mathrm{dom}(f)))\\
		=\ & \sigma(g\m)(\pi_i^{-1}(\mathrm{im}(g)\cap\mathrm{dom}(f)))\\
		=\ & \pi_i^{-1}(g^{-1}(\mathrm{im}(g)\cap\mathrm{dom}(f)))\\
		=\ & \pi_i^{-1}(\mathrm{dom}(f\circ g))\\
		=\ & \mathrm{dom}(\sigma(f\circ g)),
	\end{align*}
	and and for every $z\in \dom(\sigma(f\circ g))$,
	
	$$\sigma(f\circ g)(z)=\sigma(f)(\sigma(g)(z)).$$

	Hence,  $\sigma(f\circ g)=\sigma(f)\circ\sigma(g)$, as desired. Moreover, observe that $\sigma(\mathrm{id}_{X_i})=\mathrm{id}_{X}$.
	
	Finally, it is easy to see that $(\sigma\times c)(\Gamma(X_i)*X_i)\subset \Gamma(X)*X$ and that the following diagram is commutative
	\begin{equation*}
	\xymatrix{\Gamma(X_i)*X_i\ar[d]_-{\mathrm{ev}^{X_i}}\ar[r]^-{\sigma\times c}& \Gamma(X)*X\ar[d]^-{\mathrm{ev}^{X}}\\
			X_i\ar[r]_-{c}& X.
		}
	\end{equation*}
	This concludes the proof.
	\end{proof}
				
\section{Partial actions on function spaces}\label{pafun}	

\subsection{Construction of the partial action}\label{Subsec3.1}

Let $X$ and $Y$ be topological spaces, and let $C(X,Y)$ denote the set of continuous functions from $X$ to $Y,$ endowed with the compact-open topology $\tau_{\rm co}$. Recall that  $\tau_{\rm co}$ is the topology generated by the subbasis consisting of sets of the form $$\langle K, V \rangle= \{f \in C(X, Y) \mid f (K) \subset V\},$$  where $K\subset X$ is compact and $V\subset Y$ is open. 

We show that, whenever $X$ is compact, the space $Y$ is $\Gamma$-embedded into $C(X,Y)$. To this end, consider the embedding
	\begin{equation}\label{jm}
		c:Y\ni y\mapsto c_y\in C(X,Y),
	\end{equation}
	where $c_y$ is the constant map on $X$ with value $y$. Moreover, define the mapping  
\begin{equation}\label{eq-homoC(X,Y)}
\sigma: \Gamma(Y)\ni f\mapsto \sigma(f)\in\Gamma(C(X,Y))
\end{equation}	by
$$\sigma(f):\langle X,\mathrm{dom}(f)\rangle\ni r\mapsto f\circ r\in \langle X,\mathrm{im}(f)\rangle,$$
for $f\neq 1_{\emptyset}$, and $\sigma(f)=1_\emptyset$  otherwise.

With this notation, we obtain the following result.

\begin{theorem}\label{YencajaenC(X,Y)}
	\rm{Let $X$ be a compact space and $Y$ be a topological space. Then $(\sigma,c)$ is a $\Gamma$-embedding of $Y$ into $C(X,Y)$.}
\end{theorem}
%\begin{theorem}\label{YencajaenC(X,Y)}
%	\rm{Let $X$ be a compact space and $Y$ be a topological space. Then $Y$ is $\Gamma$-embedded into $C(X,Y)$, where $c:Y\rightarrow C(X,Y)$ and  $\sigma: \Gamma(Y)\rightarrow \Gamma(C(X,Y))$ are as in~\eqref{jm} and~\eqref{eq-homoC(X,Y)}, respectively.}
%\end{theorem}

\begin{proof}
First, we verify that $\sigma(f)\in \Gamma(C(X,Y))$. Indeed, for each compact subset $K$ of $X$ and each open subset $V$ of $Y$, we have
	\begin{align*}
		\sigma(f)^{-1}(\langle X, \mathrm{im}(f) \rangle\cap \langle K,V\rangle) & = \{r\in \langle X,\mathrm{dom}(f)\rangle\mid f(r(K))\subset V \}\\
		& =\{r\in \langle X,\mathrm{dom}(f)\rangle\mid r(K)\subset f^{-1}(V\cap\mathrm{im}(f))\}\\
		& = \langle X,\mathrm{dom}(f)\rangle\cap\langle K,f^{-1}(V\cap\mathrm{im}(f))\rangle
	\end{align*}
which shows that $\sigma(f)$ is continuous. Since $\sigma(f)^{-1}=\sigma(f^{-1})$, it follows that $\sigma(f)\in \Gamma(C(X,Y)).$

Next, we show that $\sigma$ is a semigroup homomorphism. Let $f,g\in \Gamma(Y)$. Then
	\begin{align*}
		\mathrm{dom}(\sigma(f\circ g)) & =\{r\in C(X,Y)\mid r(X)\subset \mathrm{dom}(f\circ g)\}\\
		& =\{r\in C(X,Y)\mid r(X)\subset \mathrm{dom}(g)\ \wedge\ g(r(X))\subset \mathrm{dom}(f) \}\\
		& = \{r\in C(X,Y)\mid r\in \mathrm{dom}(\sigma(g))\ \wedge\ \sigma(g)(r)\in \mathrm{dom}(\sigma(f))\}\\
		& = \mathrm{dom}(\sigma(f)\circ\sigma(g)).
	\end{align*}
	
	Moreover, for each $r\in \mathrm{dom}(\sigma(f\circ g))$,
	\begin{center}
		$\sigma(f\circ g)(r)=(f\circ g)\circ r=f\circ (g\circ r)=f\circ \sigma(g)(r)=\sigma(f)(\sigma(g)(r))$.
	\end{center}
	
	Thus, $\sigma(f\circ g)=\sigma(f)\circ\sigma(g)$, proving that $\sigma$ is a semigroup homomorphism.
	
On the other hand,	for each $(f,y)\in \Gamma(Y)*Y$, we have
		$$c(y)(X)=c_y(X)=\{y\}\subset \mathrm{dom}(f)$$
	so $c_y\in \mathrm{dom}(\sigma(f))$ and hence $(\sigma(f),c(y))\in \Gamma(C(X,Y))*C(X,Y)$. This shows that $\mathrm{im}(\sigma\times c)\subset \Gamma(C(X,Y))*C(X,Y)$.
	
	Finally, we verify that the diagram

		$$\xymatrix{\Gamma(Y)*Y\ar[d]_-{\mathrm{ev}^Y}\ar[r]^-{\sigma\times c}& \Gamma(C(X,Y))*C(X,Y)\ar[d]^-{\mathrm{ev}^{C(X,Y)}}\\
			Y\ar[r]_-{c}& C(X,Y).
		}$$ commutes. Indeed, for $(y,f)\in \Gamma(Y)*Y$, we have
	\begin{center}
		$(\mathrm{ev}^{C(X,Y)}\circ (\sigma\times c))(f,y)=\mathrm{ev}^{C(X,Y)}(\sigma(f),c_y)=f\circ c_y=c_{f(y)}=(c\circ \mathrm{ev}^Y)(f,y)$.
	\end{center}
	Therefore, the pair $(\sigma,c)$ is a $\Gamma$-embedding of $Y$ into $C(X,Y)$.
\end{proof}

\medskip
					
\textbf{From now on}, the letter $G$ will denote a topological group.

Let $Y$ be a topological space equipped with a partial action $\theta:G\rightarrow \Gamma(Y)$, and let $X$ be a compact space. By Theorem \ref{YencajaenC(X,Y)}, the space $Y$ is $ \Gamma$-embedded in $C(X,Y)$, where $c$ and $\sigma$ are as in~\eqref{jm} and~\eqref{eq-homoC(X,Y)}, respectively.  Since $\sigma(\mathrm{id}_Y)=\mathrm{id}_{C(X,Y)},$ we have 
 by Proposition \ref{motiva1}  that the map
\begin{equation}\label{eq paC(x,y)}
\hat{\theta}=\sigma\circ \theta:G\rightarrow \Gamma(C(X,Y))
\end{equation}
defines a  partial action of $G$ on $C(X,Y)$ such that $c$ is a $G$-map.

For each $g\in G$, the domain of $\hat{\theta}(g)$ is given by
	\begin{equation}\label{domindu}
		\mathrm{dom}(\hat{\theta}(g))=C(X,Y)_{g^{-1}}=\langle X, Y_{g^{-1}} \rangle  =\{f\in C(X,Y)\mid f(X)\subset Y_{g^{-1}}\}.
	\end{equation}

Consequently, the domain of $\hat{\theta}$ is
\begin{equation}\label{dom}
		G*C(X,Y):=\{(g,f)\in G\times C(X,Y)\mid f\in C(X,Y)_{g\m}\}.
	\end{equation}
	
More explicitly, the topological partial action $\hat\theta$ is 	given by
\begin{equation}\label{hat}\hat\theta: G*C(X,Y) \ni (g,f)\mapsto\theta_g\circ f\in C(X,Y).\end{equation}

In the next two results we  explore some important relations between $\theta$ and $ \hat\theta$.
				
\begin{proposition}\label{continuidad}
	\rm{Let $Y$ be a topological space, $\theta: G*Y\rightarrow Y$ be a partial action, and $X$ be a compact space. Then  $\theta$ is continuous if and only if $\hat \theta$ is continuous.}
\end{proposition}
				
\begin{proof}

Suppose that  $\theta$ is continuous. In order to prove that $\hat{\theta}$ is continuous, it suffices to show that $\hat \theta\m(\langle K, V \rangle )$ is open for every compact subset $K\subset X$  and every open subset $V\subset Y$.

Let $(g,f)\in \hat \theta\m(\langle K, V \rangle )$. Since $\{g\}\times f(K)\subset \theta^{-1}(V)$ and $\theta^{-1}(V)$ is open in $G*Y$,  it follows that for each $k\in K$ there exist neighborhoods $S_k$ of $g$ in $G$ and $T_k$ of $f(k)$ in $Y$ such that $(S_k\times T_k)\cap (G*Y)\subset \theta^{-1}(V)$.

Since  $f$ is continuous and $K$ is compact, the set $f(K)$ is compact. Therefore, there exist points  $k_1,\dots,k_n\in K$ such that $\{T_{k_i}\}_ {1\leq i\leq n}$ is an open cover of $f(K)$. Define

\begin{center}
	$W=G*C(X,Y)\cap \left(\bigcap\limits_{i=1}^nS_{k_i}\times \left\langle K, \bigcup\limits_{i=1}^{n}T_{k_i}\right\rangle\right)$.
\end{center}

Observe that $W$ is an open neighborhood of $(g,f)$ in $G*C(X,Y)$. We claim that $W\subset \hat \theta\m(\langle K, V \rangle )$. Indeed, let $(h,p)\in W$ and  $k\in K$. Then there exists $i\in\{1,\dots,n\}$ such that $p(k)\in T_{k_i}$. Consequently,

\begin{center}
	$(h,p(k))\in (S_{k_i}\times T_{k_i})\cap (G*Y)\subset \theta^{-1}(V)$,
\end{center} 
and hence $\hat \theta(h,p)(k)\in V$. Since $k\in K$ was arbitrary, we conclude that $\hat \theta(h,p)\in \langle K, V \rangle $. Therefore, $\hat \theta\m(\langle K, V \rangle )$ is open, which proves that $\hat{\theta}$ is continuous.

Conversely, suppose that $\hat \theta$ is continuous, and fix $x\in X.$
Since  the evaluation map $\mathrm{ev}_x:C(X,Y)\ni x\mapsto f(x) \in Y$  and the map $\alpha: G*Y\ni (g,y)\mapsto (g, c_y)\in G*C(X,Y)$ are  continuous,  and since $\theta=\mathrm{ev}_x\circ \hat{\theta}\circ \alpha,$ we conclude that  $\theta$ is continuous, as desired. 
\end{proof}

\begin{proposition}\label{nice}
	\rm{Let $\theta:G*Y\rightarrow Y$ be a partial action on a topological space $Y$, and let $X$ be a compact space. Then:
	\begin{enumerate}
		\item [(i)] $G*Y$ is open in $G\times Y$ if and only if $G*C(X,Y)$ is open in $G\times C(X,Y)$.\vspace{0.2cm}
		\item [(ii)] $\theta$ is nice if and only if $\hat{\theta}$ is nice.
	\end{enumerate}}
\end{proposition}

\begin{proof}
	
	(i) First, suppose that $G*Y$ is open in $G\times Y$. We show that $G*C(X,Y)$ is open in $G\times C(X,Y)$.
	
	Let $(g,f)\in G*C(X,Y)$. For each $x\in X$, choose open neighborhoods $U_x$ of $g$ in $G$ and $V_x$ of $f(x)$ in $Y$ such that $U_x\times V_x\subset G*Y$. By the compactness of $f(X)$, there are  $x_1,\dots,x_n\in X$ such that $f(X)\subset\bigcup\limits_{i=1}^nV_{x_i}$. 
		Observe that the set
	\begin{center}
		$W=\bigcap\limits_{i=1}^nU_{x_i}\times \left\langle X, \bigcup\limits_{i=1}^nV_{x_i}\right\rangle$
	\end{center}
	is an open neighborhood of $(g,f)$ in $G\times C(X,Y)$. We claim that $W\subset  G*C(X,Y)$.  Indeed, let $(h,p)\in W$ and let $x\in X$. There exists $i\in\{1,\dots,n\}$ such that $p(x)\in V_{x_i}$, so $(h,p(x))\in U_{x_i}\times V_{x_i}\subset G*Y$. This implies that $p(x)\in Y_{h^{-1}}$ for every $x\in X$, and therefore $p(X) \subset Y_{h^{-1}}$. Consequently, $(h,p) \in G*C(X,Y)$. This shows that $W\subset G*C(X,Y)$, and hence $G*C(X,Y)$ is open in $G\times C(X,Y)$.
	
	Conversely, suppose that $G*C(X,Y)$ is open. Take $(g,y)\in G*Y$, and  let $c_y:X\rightarrow Y$ be the constant map with value $y$. Then $(g,c_y)\in G*C(X,Y)$, so there exists an open neighborhood $U$ of $g$ in $G$, compact subsets $K_1,\dots,K_n$ of $X$, and open subsets $V_1,\dots,V_n$ of $Y,$ such that
$$(g,c_y)\in U\times \left(\bigcap\limits_{i=1}^n\langle K_i,V_i\rangle\right)\subset G*C(X,Y).$$ Since $y\in V:=\bigcap\limits_{i=1}^nV_i$, we obtain $U\times V\subset G*Y$, showing that $G*Y$ is open in $G\times Y$, as desired.
	
	(ii) According to Proposition \ref{continuidad}, it suffices to prove that $G*Y$ is open in $G\times Y$ if and only if $G*C(X,Y)$ is open in $G\times C(X,Y)$, which follows from (i).
	\end{proof}

\subsection{Separation properties and enveloping spaces}\label{globali}				

Let $Y$ be a topological space equipped with a partial action $\theta \colon G * Y \to Y$, and let $X$ be a compact space. 
The main purpose of this section is to investigate the relationship between certain separation properties of the enveloping spaces $Y_G$ and $C(X,Y)_G$.

	According to Proposition \ref{homosemi1} and Theorem~\ref{YencajaenC(X,Y)}, the embedding $c: Y\rightarrow C(X,Y)$, defined in \eqref{jm}, is a $G$-map. Fix $x\in X$ and consider the evaluation map $$\mathrm{ev}_x:C(X,Y)\ni f\mapsto f(x)\in Y.$$ Note that  $\mathrm{ev}_x$ is also a $G$-map. Indeed, for   $(g,f)\in G*C(X,Y)$ we have  $(g,f(x))\in G*Y$, and
				
	\begin{center}
		$\mathrm{ev}_x\big(\hat \theta(g,f)\big)=\hat \theta(g,f)(x)=\theta\big(g,f(x)\big)=\theta\big(g,\mathrm{ev}_x(f)\big).$
	\end{center}
	Moreover,  it is easy to see that the restriction $\mathrm{ev}_x|_{c(Y)}:c(Y)\rightarrow Y$ is the inverse  of $c$. 
				
	We now state the following lemma.
				
	\begin{lemma}\label{embedding}
	\rm{Let $Y$ be a topological space equipped with a topological partial action $\theta$ of  $G$, and let $X$ be a compact space. Then $Y_G$ is $G$-homeomorphic to an invariant subset $Z$ of $C(X,Y)_G$. Moreover, $Z$ is closed whenever $Y$ is Hausdorff.}
	\end{lemma}

	\begin{proof}
					
	Fix $x\in X$. As observed above, the maps $c:Y\rightarrow C(X,Y)$ and $\mathrm{ev}_x: C(X,Y)\rightarrow Y$ are equivariant.  Hence, the  maps 
	\begin{equation}\label{J}
		J: Y_G\ni [g,y] \mapsto [g,c_y]\in C(X,Y)_G
	\end{equation}
	and
	\begin{center}
	 	$K: C(X,Y)_G\ni [g,f]\mapsto[g,\mathrm{ev}_x(f)]\in Y_G$
	\end{center}  
	are also equivariant. Moreover, for each $[g,z]\in Y_G$ and each $w=[g,c_y]\in J(Y_G)$, we have
	\begin{center}
		$(K\circ J)([g,z])=K([g,c_z])=[g,z]$
	\end{center}
	and
					
	\begin{center}
		$(J\circ K)(w)=J([g,y])=[g,c_y]=w$.
	\end{center}
	
	Therefore, $J$ is a $G$-embedding. Consequently, $Y$ is $G$-homeomorphic to the invariant subset $Z=J(Y_G)$ of $C(X,Y)_G$.
	
	Now suppose that $Y$ is Hausdorff. We show that $Z$ is closed in $C(X,Y)_G$. 
	
	Let $q^{\hat{\theta}}: G\times C(X,Y)\rightarrow C(X,Y)_G$ be the quotient map corresponding to the partial action $\hat\theta$, as defined in \eqref{qo}. It suffices to show that $(q^{\hat{\theta}})\m(Z)$ is closed in $G\times C(X,Y)$.
	
	Suppose that $\big((g_\lambda, f_\lambda)\big)_{\lambda\in \Lambda}$ is a net in $(q^{\hat{\theta}})\m(Z)$, such that $\lim\ \big((g_\lambda,f_\lambda)\big)_{\lambda\in \Lambda}=(g,f)\in G\times C(X,Y)$. We claim that $(g,f)\in (q^{\hat{\theta}})\m(Z)$.  
	
	For each $\lambda\in\Lambda$, there exists $(h_\lambda,y_\lambda)\in G\times Y$ such that $[g_\lambda,f_\lambda]=[h_\lambda, c(y_\lambda)]$. Hence, there exists $k_\lambda\in G$ such that $\exists\ k_\lambda\cdot c(y_\lambda)$ and $(g_\lambda,f_\lambda)=(h_\lambda k_\lambda\m, k_\lambda\cdot c(y_\lambda))$. Moreover, since $\exists\ k_\lambda\cdot c(y_\lambda)$, it follows that  $\exists\ k_\lambda\cdot y_\lambda$ and $c(k_\lambda \cdot y_\lambda)=k_\lambda\cdot c(y_\lambda)$. Thus,
	\begin{center}
		$f_\lambda=k_\lambda\cdot c(y_\lambda)=c(k_\lambda\cdot y_\lambda)\quad$ and  $\quad f=\lim\ (f_\lambda)=\lim\ \big(c(k_\lambda \cdot y_\lambda)\big)$. 
	\end{center}  
	
	Now, for each $z\in X$ we have 
	\begin{center}
		$f(z)=\lim\ \big(c(k_\lambda \cdot y_\lambda)(z)\big)=\lim\ (k_\lambda \cdot y_\lambda).$
	\end{center}
	
	As $Y$ is Hausdorff, it follows that $f$ is the constant map with the value $y=\lim\ (k_\lambda \cdot y_\lambda)$, that is, $f=c(y)$. Consequently,
	$$q^{\hat{\theta}}(g,f)=q^{\hat{\theta}}(g,c(y))=[g,c(y)]=J([g,y])\in J(Y_G)=Z.$$
	
	Hence, $(g,f)\in (q^{\hat{\theta}})\m(Z)$, proving that $(q^{\hat{\theta}})\m(Z)$ is closed.
\end{proof} 
	
	\medskip			
				
	Below, $R_C$ denotes the equivalence relation \eqref{equiv} associated with the enveloping action of the partial action $\hat{\theta}$ of $G$ on $C(X,Y)$. Hence,  $$C(X,Y)_G=\frac{G\times C(X,Y)}{R_C}.$$  Furthermore, for a map $f\in C(X,Y)$ we set $$G^f=\{g\in G\mid (g,f)\in G* C(X,Y)\}.$$ Analogously, for each $y\in Y$ we define  $$G^y=\{g\in G\mid (g,y)\in G* Y\}.$$
				
	\begin{proposition}\label{T1T2nocontinua}
	\rm{Let  $\theta: G*Y\rightarrow Y$ be topological partial action on a topological space $Y$, and $X$ be a compact space. Then the following assertions hold:
	\begin{enumerate}
		\item [(i)]  $Y_G$ is $T_1$ if and only if $C(X,Y)_G$ is $T_1$.\vspace{0.2cm}
			
		\item [(ii)] $Y_G$ is Hausdorff if and only if $C(X,Y)_G$ is Hausdorff.
	\end{enumerate}}
	\end{proposition}

	\begin{proof}
	(i) By Lemma \ref{embedding}, it suffices to show	 that $C(X,Y)_G$ is $T_1$ whenever $Y_G$ is $T_1$. For this purpose, fix $[g,f]\in C(X,Y)_G$. We claim that $(q^{\hat{\theta}})\m(\{[g,f]\})$ is closed in $G\times C(X,Y)$, where $q^{\hat{\theta}}: G\times C(X,Y)\rightarrow C(X,Y)_G$ is the  quotient map associated with $\hat{\theta}.$ 
	
	Observe that
					
	\begin{center}
		$(q^{\hat{\theta}})\m(\{[g,f]\})=\{(gh\m,h\cdot f)\mid h\in G^f\}$.
	\end{center}
	
		Let $\big((gh_\lambda\m,h_\lambda\cdot f)\big)_{\lambda\in \Lambda}$ be a net  in $(q^{\hat{\theta}})\m(\{[g,f]\})$, and let $(k,p)\in G\times C(X,Y)$ be such that $\lim \big((gh_\lambda\m,h_\lambda\cdot f)\big)_{\lambda\in \Lambda}=(k,p)$. We shall  show that $(k,p)\in (q^{\hat{\theta}})\m(\{[g,f]\})$. 
		Take $x\in X$. For each $\lambda\in \Lambda$, we have $h_\lambda\in G^f$, then $\exists\ h_\lambda\cdot f(x)$. Moreover, since $\lim \big(h_\lambda\cdot f(x)\big)_{\lambda\in\Lambda}=p(x)$ and hence $Y_G$ is $T_1.$ In addition,  it follows that
						
		\begin{center}
		$(k,p(x))=\lim \big((gh_\lambda\m,h_\lambda\cdot f(x))\big)_{\lambda\in \Lambda}\in \overline{(q^{\theta})\m(\{[g,f(x)]\})}=(q^{\theta})\m(\{[g,f(x)]\})$.
		\end{center}
	Thus, there exists $h_x\in G^{f(x)}$ such that $(k,p(x))=(gh\m_x,h_x\cdot f(x)).$  Now, for each $x\in X$, we have $\lim\ (g h\m_\lambda)_{\lambda\in \Lambda}=k=gh\m_x$, and therefore $\lim\ (h_\lambda)_{\lambda\in\Lambda}=h_x$. Since $G$ is Hausdorff, it follows that $h_x=h_y$ for each $x,y\in X$. 
			Let $h=\lim\  (h_\lambda)_{\lambda\in \Lambda}$. Then for every $x\in X$ we have
$$(k,p(x))=(gh\m,h\cdot f(x)).$$
		Consequently,  $(k,p)=(gh\m,h\cdot f)\in(q^{\hat{\theta}})\m(\{[g,f]\})$. This completes the proof of (i).

	(ii) It is enough to show that $C(X,Y)_G$ is Hausdorff whenever $Y_G$ is Hausdorff. To this end, let us prove that $R_C$ is closed in $\big(G\times C(X,Y)\big)^2.$ 
	
	Let $(g,p), (h,q)\in G\times C(X,Y)$ be  such that $[g,p]\neq [h,q]$. There are two cases to consider.
					
	\textbf{Case 1}: $p\notin C(X,Y)_{g\m h}=\langle X,Y_{g\m h}\rangle$. Hence, there exists $x\in X$ such that $p(x)\notin Y_{g\m h}$. In particular, $[g,p(x)]\neq [h,q(x)]$. Since $Y_G=\frac{G\times Y}{R}$ is Hausdorff, there exist open subsets $A, C\subset G$ and  $B, D\subset Y$ such that
	$$\big((g,p(x)),(h,q(x))\big)\in A\times B\times C\times D\subset (G\times Y)^2\setminus R.$$

	Then $\big((g,p),(h,q)\big)\in \big(A\times \langle\{x\},B\rangle\big)\times\big(C\times \langle\{x\},D\rangle\big)$. We claim    that
	
	\begin{equation}\label{claim}
		\big(A\times \langle\{x\},B\rangle\big)\times\big(C\times \langle\{x\},D\rangle\big)\subset \big(G\times C(X,Y)\big)^2\setminus R_C.
	\end{equation}
	
	To verify \eqref{claim}, take $\big((k,r),(l,s)\big)\in \big(A\times \langle\{x\},B\rangle\big)\times\big(C\times \langle\{x\},D\rangle\big)$. If $[k,r]=[l,s]$, then $r\in C(X,Y)_{k\m l}$ and $\hat\theta_{l\m k}(r)=\theta_{l\m k}\circ r=s$. In particular, $[k,r(x)]=[l,s(x)]$ and 
	$$\big((k,r(x)),(l,s(x))\big)\in R\cap (A\times B\times C\times D)$$ which is a contradiction. Thus, $\big((g,p),(h,q)\big)\notin R_C$ and \eqref{claim} is proved.

	\textbf{Case 2}: $p\in C(X,Y)_{g\m h}.$ In this case, we have $\theta_{h\m g}\circ p\neq q$. Hence, there exists $x_0\in X$ such that $\theta_{h\m g}(p(x_0))\neq q(x_0)$. Thus, $[g,p(x_0)]\neq [h,q(x_0)]$. By the same argument used in {\bf Case 1}, there exists an open neighborhood $E$ of $\big((g,p(x)),(h,q(x))\big)$ in $\big(G\times C(X,Y)\big)^2$ such that $E\cap R_C=\emptyset$.  This shows that $R_C$ is a closed subset of $(G\times C(X,Y))^2$, and consequently $C(X,Y)_G$ is a Hausdorff space.
\end{proof}
				
	The following corollary is a direct consequence of  \cite[Proposition 3.1]{RGG} and part (ii) in Proposition \ref{T1T2nocontinua}.
	
	\begin{corollary}{\rm Let $G$ be a countable discrete group  acting partially on a compact space  $Y.$    Then  the following assertions are equivalent:
	\begin{enumerate}
	\item  [(i)] $Y_g$ is clopen for all $g\in G.$\vspace{0.2cm}
	\item   [(ii)] $Y_G$ is Hausdorff.\vspace{0.2cm}
	\item  [(iii)]  $C(X,Y)_G$ is Hausdorff.
	\end{enumerate}}
	\end{corollary}

	When $Y$ is equipped with a nice partial action of $G$, item  (ii) in Proposition \ref{T1T2nocontinua} can be proved using an alternative argument, as shown below.
				
	\begin{theorem}\label{embed}
	\rm{Let $Y$ be a topological space and let $\theta: G*Y\rightarrow Y$ be a nice partial action. Then $C(X,Y)_G$ is homeomorphic to an open subset of $C(X,Y_G)$.}
	\end{theorem}
				
	The proof of Theorem \ref{embed} follows from the following lemma.
				
	\begin{lemma}\label{openembedding}
	\rm{Let $U$ be an open subset of a topological space $Y$ equipped with a conti\-nuous action $\beta: G\times Y\rightarrow Y$ of $G$. Then, the enveloping space $U_G$, of the restriction of $\beta$ to $U$, is homeomorphic to an open subset of $Y$.}
	\end{lemma}
			
	\begin{proof}
	Set $j: U_G\ni [g,u]\mapsto \beta(g,u)\in   Y.$  Notice that $j$ is well-defined and  injective.  Let $p:G\times U\rightarrow U_G$ be the quotient map associated with the induced partial action of $\beta$ on $U.$ By \cite[(ii) Proposition 3.9]{KL}, the map $p$ is a continuous open surjection. Moreover,  since the  diagram

		\[
		\xymatrix{G\times U\ar[d]_-{p}\ar[r]^-{\mathrm{id}_G\times s}& G\times Y\ar[r]^-{\beta} &Y\\
			U_G \ar[urr]_-{j}& 
		}
		\]
	is commutative, where $s: U\hookrightarrow Y$ is the inclusion map, it follows that $j$ is open  and continuous.  Consequently, $U_G$ is homeomorphic to $j(U_G)$, which is  an open subset of $Y$.
	\end{proof}
				
	\noindent \textit{Proof of Theorem \ref{embed}}.
	\vspace{0.1cm}
Consider the map $\iota: Y\ni y\mapsto [e,y] \in Y_G$, which is an  open embedding by \cite[ Proposition 3.12]{KL}.  Since the space $Y_G$ is equipped with the continuous global action $\mu:G\times Y_G\rightarrow Y_G$ defined in (\ref{action}),  then it   follows from Proposition \ref{continuidad} that the  corresponding  global action $\hat{\mu}:G\times C(X,Y_G)\rightarrow C(X,Y_G)$ is also continuous. Moreover, observe that $C(X,Y)$ is $G$-homeomorphic to the open subset $C(X,\iota(Y))$ of $C(X,Y_G)$, where $C(X,\iota(Y))$ is endowed with the restriction of $\hat\mu$  to $C(X,\iota(Y))$.  The homeomorphism is given by the  function
	
		$$\eta: C(X,Y)\ni f\mapsto \iota\circ f \in C(X,\iota(Y)).$$
 Consequently, $C(X, Y)_G$ is homeomorphic to $C(X, \iota(Y))_G$, and Lemma \ref{openembedding} implies that $C(X, Y)_G$ is homeomorphic to an open subset of $C(X, Y_G)$.
	\qed
				
	\begin{corollary}\label{regme}
	\rm{Let $\theta:G*Y\rightarrow Y$ be a nice partial action on a topological space  $Y$, and let $X$ be a compact space. Then  the following assertions hold:
	\begin{enumerate}
	\item [(i)] $Y_G$ is metrizable if and only if $C(X,Y)_G$ is metrizable.\vspace{0.2cm}
	\item [(ii)] $Y_G$ is regular if and only if $C(X,Y)_G$ is regular.
	\end{enumerate}}
	\end{corollary}
				
	\begin{proof}
	(i) Suppose that $Y_G$ is metrizable. Since $X$ is compact, the compact-open topology on $C(X,Y_G)$ coincides with the topology induced by the uniform metric; hence $C(X,Y_G)$  is metrizable. Therefore, by Theorem \ref{embed}, the space $C(X,Y)_G$ is also metrizable.  Conversely, if $C(X,Y)_G$ is metrizable, then Lemma \ref{embedding} shows that $Y_G$ is  metrizable as well.
	
	(ii) If $Y_G$ is regular, then $C(X,Y_G)$ is regular. Consequently, by Theorem \ref{embed}, the space $C(X,Y)_G$ is regular. Conversely, the regularity of $C(X,Y)_G$, together with Lemma \ref{embedding}, yields that $Y_G$ is regular.
	\end{proof}
		
	\begin{remark}{\rm The  following commutative diagram states 
			   a relation between the enveloping maps	$\mu^\theta$ and $\mu^{\hat\theta}$  of $\theta$ and $\hat\theta$ respectively, and the action $\hat{\mu^\theta}$ of $G$ on  $C(X, Y_G).$ 
						
	\[\begin{tikzcd}
		G\times C(X, Y_G) & C(X, Y_G) \\
		G\times C(X, Y)_G & C(X, Y)_G \\
		G\times Y_G & Y_G
		\arrow["\hat{\mu^\theta}", from=1-1, to=1-2]
		\arrow [ "{\rm id}_G\times\xi", from=2-1, to=1-1]
		\arrow["\xi  ",  from=2-2, to=1-2]
		\arrow["\mu^{\hat\theta}", from=2-1, to=2-2]
		\arrow [ "{\rm id}_G\times J ", from=3-1, to=2-1]
		\arrow["\mu^{\theta}",  from=3-1, to=3-2]
		\arrow[ "J ",  from=3-2, to=2-2]
	\end{tikzcd},\]
	where $J$  is the continuous map  given by \eqref{J},  $\iota^\theta$ is defined in   \eqref{iota} and $$\xi: C(X,Y)_G\ni [g, \psi]\mapsto  \mu^\theta_g\circ\iota^\theta\circ \psi\in C(X,Y_G)$$  is the   embedding given  the proof of Theorem \ref{embed}.}
\end{remark}

\subsection{A working example}\label{examp} Let $\theta$ be a nice partial action on a topological space $Y$. By Theorem \ref{embed}, we know that $C(X,Y)_G$ is $G$-homeomorphic to an open subset of $C(X,Y_G).$ Below we shall present an example where   $C(X,Y)_G$  and $C(X,Y_G)$ are $G$-homeomorphic.

\begin{example}
	\rm{Let \( G = \mathbb{Z} \) equipped with the discrete topology. We consider the global action on $\R$ given by
	$$\beta: G\times \mathbb{R} \ni (n,t)\mapsto n+t\in \R.$$
Let $Y=(0,\infty)$ be considered as a subspace of $\mathbb{R},$ and  $\theta=(Y_n, \theta_n)_{n\in \Z}$ the restriction of $\beta$ to $Y$. Specifically,
	
	\begin{center}
		$Y_n = (0,\infty)\cap (n,\infty) = (\max\,\{0,n\},\infty)\quad$ and $\quad \theta_n :Y_{-n}\ni y\mapsto n+y\in Y_n.$
	\end{center}	
Since $Y$ is open we have, from Example~\ref{induc}, that $\theta$ is a nice  partial action.}
\end{example}

\begin{proposition}
The enveloping space $Y_G$ of the partial action $\theta$ defined above is $G$-homeomorphic to $\mathbb{R}$ endowed with the global action $\beta: G\times \mathbb{R} \ni (n,t)\mapsto n+t\in \R.$
\end{proposition}		 
			
\begin{proof}

The globalization $Y_G$  is the quotient space $Y_G = (G\times Y)/{R},$
where $R$ is the equivalence relation introduced in  \eqref{equiv}. Explicitly, $(n,y) R (m,z)$ if and only if $y\in  Y_{m-n}$ and $z=\theta_{n-m}(y)$; equivalently, $y\in  (\max\,\{0,m-n\},\infty)$ and $m+z=n+y.$ 

Define the map
$$\Phi: \Z\times Y\ni (n, y)\mapsto n+y\in \R.$$ 

This map is  continuous and surjective. Moreover, $\Phi(n,y)=\Phi(m,z)$ if and only if  $(n,y) R (m,z)$. Hence, the induced map $Y_G\ni [n,y]\mapsto n+y\in \R$ is a $G$-homeomorphism.
\end{proof}
									
Let $X$ be a topological space. Consider the space $C(X,Y)$ endowed with the partial action 
	\[
		\hat\theta: G*C(X, Y)\ni (n,f)\mapsto \theta_n\circ f\in C(X, Y),
	\]
induced by $\theta$. Note that $\hat\theta$ is a nice partial action by  Proposition \ref{nice}.

For each $n\in G$, we have
	\[
	C(X,Y)_n = \{ f\in C(X,Y)\mid f(x)> \max\,\{0,n\}\ \text{for all } x\in X\}.
	\]

Using this notation, we can now state the following result.

\begin{theorem}
\rm{The enveloping space $C(X,Y)_G$ of the partial action $\hat\theta$ is $G$-homeomorphic to the space $C(X,\R)$.}
\end{theorem}

\begin{proof}
Consider the open continuous map
	\[
	\phi : G\times C(X,Y)\ni (n,f)\mapsto f+n\in C(X,\mathbb R), 
	\]
where $(f+n)(x)=f(x)+n,$ for all $x\in X.$  

Observe that $\phi$ is surjective. Indeed, given $F\in C(X,\mathbb R),$ we  set
$m=\min_{x\in X}F(x),$  which exists because $X$ is compact. Let $n=\lfloor m\rfloor-1$.
Then the map $f:=F-n$ belongs to $C(X,Y)$, and $\phi(n,f)=f+n=F$. 

Now, for $(n,f), (m,g)\in G\times C(X,Y)$ we have $\phi(n,f)=\phi(m,g)$ if and only if $f+n=g+m$; that is, 
	\begin{center}
		$\phi(n,f)=\phi(m,g)$ if and only if  $(n,f)R^{\hat\theta}(m,g)$.
	\end{center}
Therefore, $\phi$ induces a homeomorphism
 	\begin{center}
 		$\Phi:C(X,Y)_G\ni[n,f]\mapsto n+f\in C(X,\mathbb R)$.
 	\end{center}
 Moreover, since $\Phi$ is a $G$-homeomorphism, we conclude that $C(X,\mathbb{R})$ and $C(X,Y)_G$ are $G$-homeomorphic.
\end{proof}

\subsection{Globalization and ANEs} 
				
Let $X$ be a topological space equipped with a continuous global action of a topological group $G$. Recall that $X$ is called an {\it equivariant absolute neighborhood extensor} (respectively, an {\it equivariant absolute extensor}) if for every metrizable space $Z$ endowed with a continuous global action of $G$, and for every closed $G$-subset $A$ of $Z$, any $G$-map $f:A\rightarrow X$ extends to a $G$-map $F:U\rightarrow X$, where $U$ is an invariant neighborhood of $A$ in $Z$ (respectively, $U=Z$). In this case, we say that $X$ is a $G$-ANE (respectively, a $G$-AE).

Some of the partial actions studied in this work have already been considered in the global setting, along with their relationship with $G$-AEs and $G$-ANEs, as illustrated in the following example.

\begin{example}$~$
\rm{
\begin{enumerate}
		\item[(i)] Let $G$ be a compact Lie group and $X$ be a metrizable $G$-space equipped with a continuous global action $\theta:G\times X\rightarrow X$. It follows from Proposition \ref{homosemi1} (see also \cite[Theorem 3.2]{MPR}) that the map
		\[
		\hat{\theta}:G\times\mathcal{K}(X)\ni (g,A)\mapsto \theta(\{g\}\times A)\in\mathcal{K}(X)
		\]
		is a continuous global action of $G$ on $\mathcal{K}(X)$. A fundamental result proved in \cite[Theorem 1.1]{A4} states that $X$ is locally continuum-connected (respectively, connected and locally continuum-connected) if and only if $\mathcal{K}(X)$ is a $G$-ANE (respectively, a $G$-AE).
			
		\item[(ii)] Let $G$ be a compact group and $L$ be a Banach space equipped with a continuous global action $\theta$ such that $\theta_g:L\rightarrow L$ is linear for each $g\in G$. If $V\subset L$ is a closed invariant convex subset, then it follows from \cite[Theorem 3.2]{A2} that $V$ is a $G$-AE. The proof of this result makes use of the following equivariant map:
		\[
		S:C(G,V)\rightarrow V,\qquad S(f)=\int_G f\, dm,
		\]
		where $m$ is the Haar measure on $G$, and the integral is taken in the sense of \cite[Definition 3.26]{WR}. Here, $C(G,V)$ is equipped with the continuous global action $\hat{\theta}$ defined in \eqref{eq paC(x,y)}.
\end{enumerate}
}
\end{example}

\begin{proposition}
Let $X$ be a compact space, and let $Y$ be a topological space equipped with a continuous global action $\theta$ of a compact group $G$. 
Then $Y$ is a $G$-AE (respectively, a $G$-ANE) if and only if $C(X,Y)$ is a $G$-AE (respectively, a $G$-ANE), where $C(X,Y)$ is endowed with the continuous global action $\widehat{\theta}$ defined in~\eqref{eq paC(x,y)}.
\end{proposition}
\begin{proof}
First, fix $x\in X$, and consider the evaluation map $\mathrm{ev}_x: C(X,Y)\ni f\mapsto f(x)\in Y$ and the embedding $c:Y\ni y\mapsto c_y\in C(X,Y)$ defined in (\ref{jm}). We have already shown that both $c$ and $\mathrm{ev}_x$ are equivariant. The subset $c(Y)\subset C(X,Y)$ is invariant, and the map $r=c\circ \mathrm{ev}_x:C(X,Y)\rightarrow c(Y)$ is an equivariant retraction, as $r(f)=f$ for all $f\in c(Y)$. Consequently, the space $c(Y)$ is a $G$-AE (respectively, a $G$-ANE) whenever $C(X,Y)$ is. Since $Y$ and $c(Y)$ are $G$-homeomorphic, it follows that $Y$ is a $G$-AE (respectively, a $G$-ANE) provided that  $C(X,Y)$ is.
			
		Conversely, assume that $Y$ is a $G$-ANE. We show that $C(X,Y)$ is a $G$-ANE.  The $G$-AE case is analogous.  Let $Z$ be a metrizable space equipped with a continuous global action of $G$, and let $\phi: A\rightarrow C(X,Y)$ be a $G$-map defined on a closed invariant subset $A\subset Z$. Since $X$ is compact, the map $\Phi:A\times X\ni (a,x)\mapsto\phi(a)(x)\in Y$ is continuous (see, e.g., \cite[Theorem 3.4.3]{RE}).	
			
		Equip $Z\times X$ with the continuous global action $g\cdot (z,x)=(g\cdot z,x),$ for all $g\in G$ and $(z,x)\in Z\times X.$ Then $A\times X$ is closed and invariant, and $\Phi$ is equivariant. Indeed, for $g\in G$ and $(a,x)\in A\times X$ we see that
		\begin{center}
			$\Phi(g\cdot a,x)=\phi(g\cdot a)(x)=\hat{\theta}(g,\phi(a))(x)=\theta(g,\phi(a)(x))=\theta(g,\Phi(a,x))$.
		\end{center}
Since $Y$ is a $G$-ANE, there exist an invariant neighborhood $U$ of $A\times X$ in $Z\times X$ and a $G$-map $\psi: U\rightarrow Y$ extending $\Phi$. Since  $G$ and $X$ are compact spaces, we may assume that $U=V\times X$ for some invariant neighborhood $V$ of $A$ in $Z$. By \cite[Theorem 3.4.3]{RE}, the map $\Psi:V\rightarrow C(X,Y)$ defined by
		\begin{center}
			$\Psi(v)(x)=\psi(v,x),\quad v\in V,\quad x\in X,$
		\end{center}
		is continuous. It remains to show that $\Psi$ is a $G$-map and extends $\phi$.
			For $a\in A$ and $x\in X$ we have
		\begin{center}
			$\Psi(a)(x)=\psi(a,x)=\Phi(a,x)=\phi(a)(x),$
		\end{center}
		showing that $\Psi(a)=\phi(a)$, and thus $\Psi$ extends $\phi$. Moreover, for each $v\in V$, $x\in X$, and $g\in G$ we see that
		\begin{center}
			$\Psi(g\cdot v)(x)=\psi(g\cdot v,x)=\theta(g,\psi(v,x))=\theta(g,\Psi(v)(x))=\hat{\theta}(g,\Psi(v))(x)$
		\end{center}
		which implies that $\Psi(g\cdot v)=\hat{\theta}(g,\Psi(v))$. Therefore, $\Psi$ is equivariant. This completes the proof.
\end{proof}

When \(G\) is trivial, the notions of \(G\)-ANE and \(G\)-AE reduce to those of ANE and AE, respectively. 
For further background on the theory of retracts, we refer the reader to \cite{SH,KS}.

Under the hypotheses of Theorem \ref{embed}, we prove in the following result that the property of being an ANE is preserved between $Y_G$ and $C(X,Y)_G$; more precisely, $C(X,Y)_G$ is an ANE if and only if $Y_G$ is.

\begin{theorem}\label{ane}
\rm{Let $\theta$ be a nice partial action  of $G$ on a topological space $Y$, and let $X$ be a compact metrizable space. Then the following statements are equivalent:
						
	\begin{enumerate}
	\item [(i)] $Y$ is an ANE.\vspace{0.2cm}
	\item [(ii)] $Y_G$ is an ANE.\vspace{0.2cm}
	\item [(iii)] $C(X,Y)_G$  is an ANE.\vspace{0.2cm}
	\item [(iv)] $C(X,Y)$ is an ANE.
	\end{enumerate}}

\end{theorem}
				
\begin{proof}
	
(i) $\Rightarrow$ (ii) Consider the open embedding $\iota: Y\rightarrow Y_G$ defined in (\ref{iota}), and let $\mu: G\times Y_G\rightarrow Y_G$ be the enveloping  action of $G$ on $Y$ defined in \eqref{action}. Observe that $Y=\bigcup_{g\in G}\mu_g(\iota(Y))$. Moreover, for each $g \in G$, the space $\mu_g(\iota(Y))$ is homeomorphic to $Y$. Hence, by \cite[Chapter II, Theorem 17.1]{SH}, if $Y$ is an ANE then $Y_G$ is also an ANE.

To prove (ii) $\Rightarrow$ (iii), suppose that $Y_G$ is an ANE. From \cite[Chapter VI, Theorem 2.4]{SH} we have  that $C(X,Y_G)$ also is ANE. Furthermore, by Theorem \ref{openembedding}, the enveloping space $C(X, Y)_G$ is homeomorphic to an open subset of $C(X, Y_G)$. Thus, by \cite[Chapter II, Theorem 6.1]{SH}, we conclude that $C(X, Y)_G$ is an ANE.

(iii) $\Rightarrow$ (iv) Now suppose that $C(X, Y)_G$ is an ANE. By Proposition~\ref{nice}, the partial action $\hat{\theta}$ on $C(X, Y)$ is nice. Then the map $\iota^{\hat{\theta}} : C(X, Y) \rightarrow C(X, Y)_G$ is an open embedding. Hence, by \cite[Chapter VI, Theorem 2.4]{SH}, we conclude that $C(X, Y)$ is an ANE. 

Finally, to prove (iv) $\Rightarrow$ (i), suppose that $C(X, Y)$ is an ANE. Since $Y$ is homeomorphic to a retract of $C(X, Y)$, using \cite[Chapter II, Proposition 5.2]{SH} we obtain that $Y$ is an ANE. 
\end{proof} 
We finish our work with the next.
\begin{example}{\rm

Consider the group  $G=\rm{GL}(2;\mathbb{R})$ acting  partially   on $\mathbb{R},$  where 

	\begin{center}
		$G*\mathbb R=\left\{(g,x)\in G\times \mathbb R\mid g=\begin{pmatrix}
			a & b \\
			c & d 
		\end{pmatrix} \text{and } cx+d\neq 0\right\}$
	\end{center}
	and 
	$$\theta:G*\mathbb R\ni (g,x)\mapsto \frac{ax+b}{cx+d}\in \mathbb R.$$
It follows by \cite[Lemma 3.15]{KL}  that $\mathbb{R}_G$ is homeomorphic to $\mathbb{S}^1=\{x\in \mathbb{R}^2 \mid \|x\|=1\}$. By the Dugundji Extension Theorem \cite[Theorem 6.1.1]{KS}, we know that $\mathbb{R}$ is an AE. Hence, it follows from Theorem \ref{ane} that, for any compact metrizable space $X$, the space $C(X,\mathbb{R})_G$ is an AE.}

\end{example}

\section{Acknowledgments}
The first named author was supported by grant 829945 from SECIHTI (M\'exico). The third named author was supported by PRPI da
Universidade de São Paulo, process n°: 22.1.09345.01.2.

\end{document}